\documentclass[12pt, english]{amsart}

\usepackage{amsfonts}
\usepackage{amssymb, latexsym, amsthm, amsmath, verbatim}
\usepackage{graphicx}
\usepackage{enumerate}
\usepackage[all,cmtip]{xy}
\usepackage{color}
\usepackage{float}
\usepackage{csquotes}
\usepackage{tikz}
\usetikzlibrary{arrows}
\restylefloat{table}

\usepackage{fullpage}

\theoremstyle{theorem}

\newtheorem{thm}{Theorem}[section]
\newtheorem{lem}[thm]{\textbf Lemma}
\newtheorem{cor}[thm]{Corollary}
\newtheorem{prop}[thm]{\textbf Proposition}

\newtheorem{con}[thm]{Conjecture}

\theoremstyle{definition}

\newtheorem{defn}[thm]{Definition}

\newtheorem{exmp}[thm]{\textbf{Examples}}

\newtheorem{prob}[thm]{Problem}

\numberwithin{equation}{section}

\theoremstyle{remark}

\newcommand{\R}{\mathbb{R}}

\newcommand{\Z}{\mathbb{Z}}

\DeclareMathOperator{\sgn}{sgn}

\makeatletter
\def\Ddots{\mathinner{\mkern1mu\raise\p@
\vbox{\kern7\p@\hbox{.}}\mkern2mu
\raise4\p@\hbox{.}\mkern2mu\raise7\p@\hbox{.}\mkern1mu}}
\makeatother

\usepackage[T1]{fontenc}
\usepackage{babel}
\begin{document}
\author[C. Coscia]{Christopher Coscia}
\address{Department of Mathematics, 6188 Kemeny Hall, Dartmouth College, Hanover, NH,
03755-3551.}
\email{christopher.s.coscia.gr@dartmouth.edu}

\author[J. DeWitt]{Jonathan DeWitt
}
\address{Department of Mathematics and Statistics, Haverford College, 370 Lancaster Avenue, Haverford, PA, 19041, USA.}
\email{jdewitt@haverford.edu}

\title{Locally Convex Words and Permutations}


\begin{abstract}
We introduce some new classes of words and permutations characterized by the second difference condition $\pi(i-1) + \pi(i+1) - 2\pi(i) \leq k$, which we call the $k$-convexity condition. We demonstrate that for any sized alphabet and convexity parameter $k$, we may find a generating function which counts $k$-convex words of length $n$.  We also determine a formula for the number of 0-convex words on any fixed-size alphabet for sufficiently large $n$ by exhibiting a connection to integer partitions.  For permutations, we give an explicit solution in the case $k = 0$ and show that the number of 1-convex and 2-convex permutations of length $n$ are $\Theta(C_1^n)$ and $\Theta(C_2^n)$, respectively, and use the transfer matrix method to give tight bounds on the constants $C_1$ and $C_2$.  We also providing generating functions similar to the the continued fraction generating functions studied by Odlyzko and Wilf in the ``coins in a fountain'' problem. 
\end{abstract}



\maketitle

\section{Introduction}
In this paper we investigate permutations $\pi\in S_n$ and words (functions $f: [n] \to [p]$) that obey the condition:
\[
 \pi(i-1) + \pi(i+1) - 2\pi(i) \le k \hspace{.5cm} \text {for all } i\in [2,...,n-1],
\]
and
\[
f(i-1) + f(i+1) - 2f(i) \le k \hspace{.5cm} \text {for all } i\in [2,...,n-1],
\]
respectively, where $k\in \Z_{n\ge 0}$ and $[n]=\{1,2,...,n\}$. We refer to these as $\emph{locally convex}$ permutations and words \emph{with respect to $k$}, or $k$-convex permutations and words. Geometrically, 0-convex permutations and words, such as 1342 or 14444322, respectively, are those such that when the permutation or word entries are plotted against their positions and consecutive entries are connected by straight line segments, the area under the plot is convex. This description is different than that presented in \cite{Albert11}, in which the authors consider the polygon \emph{enclosed} by the plotted points.  Intuitively, an increase in the parameter $k$ represents a relaxation of this condition.

The study of these permutations arises from a problem in a graduate course taught by Jamie Radcliffe and brought to our attention by Jessie Jamieson.  The original problem is stated as follows: 

\begin{prob}
Let $\sigma \in S_n$. If, for all $i\in [n-1]$, we have:
\[
\sigma(i+1)\le \sigma(i)+1
\]
then $\sigma$ is a \emph{slow riser}. Let $Slow_n$ be the number of slow risers in $S_n$. What is $Slow_n$?
\end{prob}
$Slow_n$ is $2^{n-1}$ since for any slow riser $\sigma$ of length $n-1$, there are exactly two places in the one line notation for $\sigma$ in which we can insert $n$ to form a slow riser of length $n$; $n$ may be inserted immediately after $n-1$ or at the beginning of the permutation. This \enquote{extension} map is $1$ to $2$ as removing $n$ from a valid permutation of length $n$ gives us a valid permutation of length $n-1$.

We may abstract the property given above by noticing that it could equivalently be stated that $\sigma\in S_n$ is a slow riser if the first differences of $\sigma$ are bounded above by one. We can generalize this problem by selecting some value other than one (call it $k$) to work with, but the previous argument still works as long as we pick a suitable extension procedure, so we consider instead second differences.  We thank Bill Kay for suggesting this generalization. 

First, we discuss locally convex words, for which, in the case $k=0$, we may enumerate by exploiting a bijection with pairs of integer partitions. In addition, we demonstrate that it is possible to derive a generating functions for locally convex words for any $k$. We then study the case of $k$-convex permutations, solving the case $k = 0$ and deriving asympotic estimates for $k$ equal to $1$ and $2$. We also give generating functions for the cases $k=1,2$. Further, we are interested in seeing whether or not a Marcus-Tardos type result holds for $k$-convex permutations, as we are able to show that it does for the cases $k=1$ and $k=2$. 

The aforementioned generating functions which enumerate 1- and 2-convex permutations are quite similar to the generating function that describes the number of solutions to the ``coins in a fountain'' problem, as described in \cite{Odlyzko88}. As such, further simplifications to our generating functions will be closely related to the study of the generating function for fountains, which admits a similar continued fraction, but does not seem to have a more usable form.

\section{Locally Convex Words}
We wish to count functions $f:[n]\rightarrow [p]$ such that for $i\in \{2,...,n-1\}$ we have 
\begin{equation}\label{eq:convex}
f(i-1) + f(i+1) - 2f(i)\le k
\end{equation}
for some $k\in \Z_{\ge 0}$. We say that such a function is \emph{locally convex with parameter $k$}.  Notice that convex is an appropriate word to use here, as this stipulates a bound on the growth of the first differences:

\[
f(i+1) - f(i) \leq f(i) - f(i-1) + k \implies f(i-1) + f(i+1) - 2f(i) \leq k
\]

Fix $p$ and $k$. Let $f_n(a,b)$, for $a,b\in [p]$, be the number of functions $f:[n]\rightarrow [p]$ such that $f(1)=a$ and $f(2)=b$ that also obey the convexity condition \eqref{eq:convex}.  We will determine the generating function:
\[
F(a,b)=\sum_{n\ge 2} f_n(a,b)x^n
\]
and when all of these generating functions have been determined, we will simply sum them in order to find the generating function for all such functions. That is,

\[
G_{k,p}=1+px+\sum_{a,b\in [p]} F(a,b)
\]

We will now give a method of determining $F(a,b)$ given some fixed $p$ and $k$ using the Transfer Matrix method, which is described in  \cite{Stanley11}, Section 4.7.  For a description and examples of the transfer matrix method as used for enumerating permutations, see \cite{Kitaev11}, \cite{Mansour01}, and \cite{Vatter06}.

These functions are in bijection with sequences $a_1,...,a_n$ where $a_i=f(i)$, which we will discuss below.  
To begin, we claim that for $n\ge 3$, we have:
\begin{equation}\label{eq:recurrence}
f_n(a,b)=\sum_{i\le k+2b-a} f_{n-1}(b,i)
\end{equation}
Given $a_1,...,a_n$, we know that $a_2,...,a_n$ is unique and counted by $f_{n-1}(a_2,a_3)$.  If $a_1,...,a_{n-1}$ is a sequence such that $a_1=b$, then we see that this sequence has a unique extension to a sequence beginning with $a$ if and only if $a+a_2-2a_1\le k$, or equivalently, $a_2\le k+2a_1-a$. Then, summing over all possible $a_2$, we find the relation above. 

We now claim that
\[
F(a,b)=x\left(\sum_{i\le k+2b-a} F(b,i)\right)+x^2.
\]
To see this, we multiply by $x^n$ and sum over \eqref{eq:recurrence} for $n\ge 3$. We have
\[
\sum_{n\ge 3} f_n(a,b)x^n=\sum_{n\ge 3}\sum_{i\le k+2b-a} f_{n-1}(b,i)x^n.
\]
Reversing the order of summation on the right and changing the indexing, we find:
\[
\sum_{n\ge 3}f_n(a,b)x^n=\sum_{i\le k+2b-a}x\sum_{n\ge 3}f_{n}(b,i)x^n.
\]
Now note that $[x^2]F(a,b)=1$ so we can add $x^2$ to both sides and conclude \eqref{eq:recurrence}. Furthermore, note that we have $p^2$ such in $p^2$ unknowns. To show that it is possible to find a generating function, it will suffice to check that such equations have a unique solution. 

Note that we can rearrange the above equation as follows:
\begin{equation}\label{eq:reqrer}
x^2=F(a,b)-x\left(\sum_{i\le k+2b-a} F(b,i)\right).
\end{equation}

For ease of notation, we order the $F(a,b)$ lexicographically, i.e. $F(a,b)<F(c,d)$ if $a<c$ or $a=c$ and $b<d$. Given this order, let $F_i$ be the $i$th ordered element. 

Let $A_{ij}$ be the matrix described by the $p^2$ equations above, ordered using the aforementioned linear order. Note that the diagonal entries of $A$ are either $1$ or $1-x$ and that the off diagonal entries are either $0$ or are divisible by $x$. We wish to show that the coefficient of $x^0$ in the determinant of $A$ is $1$. Let $q$ be the number of diagonal entries of $A$ equal to $1-x$. Then,

\begin{align*}
\det(A)&=\sum_{\sigma\in S_n}\sgn(\sigma)a_{1,\sigma(1)}...a_{2,\sigma(2)}\\
&=\sgn(e)a_{1,1}...a_{n,n}+\sum_{\sigma\in S_n-\{e\}}\sgn(\sigma)a_{1,\sigma(1)}...a+{2,\sigma(2)}\\
&=(1-x)^q+\sum_{\sigma\in S_n-\{e\}}\sgn(\sigma)a_{1,\sigma(1)}...a+{2,\sigma(2)}\\
&=1+\sum_{n\ge1} -x^n{q\choose n}+\sum_{\sigma\in S_n-\{e\}}\sgn(\sigma)a_{1,\sigma(1)}...a+{2,\sigma(2)}\\
\end{align*}
By our previous observations, note that $x$ divides both of the summations above, so $\det(A)$ has the form $1+xy$ for some $y\in \R[[x]]$. Note that the leading term of $x$ is $1$, and thus it is a unit in $\R[[x]]$. Since a matrix with elements in a commutative ring with identity, $R$, is invertible if and only if its determinant is a unit in $R$, and so $A$ is invertible in $\R[[x]]$ as $\det(A)$ is a unit.  This means the above system of equations has a unique solution. To arrive at the correct generating function, one must then add back in permutations of length $1$ and $0$ that do not fit into the scheme above.

As an example we can calculate that 
\begin{align*}
G_{0,3} &= 1+3x+9x^2+16x^3+20x^4 + 21x^5/(1-x) \\
&= 1+3x+9x^2+16x^3+20x^4 + 21x^5 + 21x^6 + 21x^7 + \ldots
\end{align*}

Notice that the number of such permutations as $n$ becomes large stays constant at the value of $21$.  Perhaps somewhat surprisingly, this is true for all values of $p$ when $k = 0$; this is our next result.

\begin{thm}
Let $g_{0, p}(n) = [x^n]G_{0, p} = \#\{f: [n] \to [p] \, | \, f(i-1) + f(i+1) - 2f(i) \leq 0 \, \forall \, i \in \{2, 3, \ldots, n-1\}\}$, then for $n > 2(p-1)$,
\[
g_{0, p}(n) = \sum_{m=1}^p \left(\sum_{j=0}^{m-1} a(j)\right)^2
\]
where $a(j)$ is the number of integer partitions of $j$.
\end{thm}

\begin{proof}
We rewrite this sum as 
\[
g_{0, p}(n) = \sum_{m=1}^p h(m)
\]
and see that $h(m)$ can be interpreted as the cardinality of $H_m$, which we define as the set of 0-convex words of a fixed length at least $2p-1$ on the alphabet $[p]$ that attain a maximal value of $m$.  Now given a word $w \in H_m$, we will show that every element in $H_m$ is determined uniquely by an ordered pair $\{w_1, w_2\}$ of integer partitions of $m-w_f$ and $m-w_{\ell}$ for $w_f$ and $w_{\ell}$ the first and last entries of $w$, respectively, with possible values between 1 and $m$, which would mean that $h(m) = (\sum_{j=0}^{m-1} a(j))^2$.

Fix $m \leq p$ and let $w_1, w_2$ be two integer partitions of $m_1, m_2 < m$, respectively.  Write $w_1$ as a sequence with elements corresponding to the parts of the integer partitions written in increasing order.  Now construct the word $pf$, where for all $i = 1, 2, \ldots, length(w_1)$, 
\[
pf(i) = m - \left(\sum_{j = i+1}^{length(w_1)} w_1(j)\right).
\]
The result is a 0-convex word of length at most $m_1$ whose first entry is $m - m_1 \geq 1$; it is strongly increasing and 0-convex (the second difference condition is satisfied because the sequence of first differences is weakly decreasing while the entries themselves are increasing, by construction).  Similarly write $w_2$, the sequence of partial sums of the partition of $m_2$, also in decreasing order, and form the word $sf$ where 
\[
sf(i) = m - \left(\sum_{j = 1}^{i} w_2(j)\right).
\]
Then $sf$ is 0-convex and strongly decreasing, with final element $m - m_2$.  For $c \geq 0$, form a 0-convex word of length $2p-1+c$ that attains a maximum value of $m$ by appending $2p - 1 + c - {length}(pf) - {length}(sf) \geq 2p-1+c-m_1-m_2\geq c+1$ copies of $m$ to $pf$ and then further appending $sf$ to the result.  To show that the concatenation of these three 0-convex sequences is also 0-convex, we must check the boundaries between the subwords.  Clearly $2m \geq m + pf(length(w_1))$ and $2m \geq m + sf(1)$; notice also that by the construction of the words $pf$ and $sf$, the difference between the final two entries of $pf$ is less than the difference between $m$ and the final entry of $pf$, and because this sequence is increasing the second difference condition is satisfied.  A similar argument holds for the transition between the $m$ plateau and $sf$, so we conclude that this construction maps the pair $\{w_1, w_2\}$ into $(pf)m\ldots m (sf)$, a 0-convex word in $H_m$.

Next, fix $c \geq 0$ and let $w$ be a word of length $2p-1+c$ that attains a maximum value of $m$ ($w \in H_m$).  We know due to the convexity condition that $w$ may be written as some subset of a strictly increasing subword followed by a plateau consisting of the value $m$ and finally by a strictly decreasing subword; this ensures that all appearances of $m$ in $w$ are consecutive.  Further, we know that $m$ appears at least once.  We can then write $w = (pf)M(sf)$ where $pf$ and $sf$ are 0-convex, $pf$ is strongly increasing, $sf$ is strongly decreasing, and $M$ is a string of $m$'s of length $2p-1+c - length(pf) - length(sf) \leq c+1 < 0$ (hence $w$ attains its maximum).  Then $pf$ encodes a unique integer partition $w_1$ of $m - pf(1)$; the partition is given by the first differences of the entries in $pf$ and $m - pf(length(w_1))$ (the fact that this sequence is monotone decreasing proves uniqueness).  As a similar argument can be made for $sf$ into the partition $w_2$ of $m - sf(length(w_2))$, we can now map the words $w = (pf)M(sf)$ into the pairs $\{w_1, w_2\}$ where $w \in H_m$.  Thus, $h(m) = |H_m| = |\text{partitions of }0, 1, 2, \ldots, m-1|^2 = (\sum_{j=0}^{m-1} a(j))^2$ and the result follows.
\end{proof}

\begin{exmp} 
Let $p = 3, w_1 = \{1\}$ (a partition of 1), $w_2 = \{1, 1\}$ (a partition of 2), then $pf = 2, sf = 21, w = 23321$. \\

Let $p = 8, w_1 = \{1, 1, 2, 3\}$ (a partition of 7), $w_2 = \{2, 4\}$ (a partition of 6), then
$pf = 1467, sf = 62, w = 146788888888862$. \\

Let $p = 5, w_1 = \{1, 1, 1, 1\}$ (a partition of 4), $w_2 = \{3\}$ (a partition of 3), then $pf = 1234, sf = 2, w = 123455552$.
\end{exmp}

The sequence $\{g_{0, p}(2p-1)\}_p$ begins $1, 5, 21, 70, 214, 575, 1475, 3500, 7469, \ldots$

\section{Locally Convex Permutations}
\begin{defn}Let $\pi$ be a permutation of length $n$ $(\pi: [n] \to [n]$ is a bijection).  We say that $\pi$ is \emph{$k$-convex} for a nonnegative integer $k$ if it obeys the following: \\  \\ \centerline{$\pi(i-1) +\pi(i+1) - 2\pi(i) \leq k$ for all $i \in \{2, 3, \ldots, n-1\}$.}  \\

If the statement is true for $k = 0$, we call $\pi$ \emph{perfectly convex}, and if the statement is true for $k = 2$ (and, as a result, for $k = 1$ and $k = 0$), we say that $\pi$ is \emph{strongly convex}.
\end{defn}

Throughout, we will write $\pi$ in one line notation as $\pi(1)\pi(2)\pi(3)\ldots\pi(n)$.  We are interested in enumerating the $k$-convex permutations of length $n$, which we define as $f_k(n)$.  First we will establish a few elementary properties of permutation convexity. Techniques similar to ours for enumeration and estimating growth rates may be found in \cite{Albert06} and \cite{Branden05}.

\begin{prop}
If $\pi$ is a $k$-convex permutation of length $n$, then $\pi^{R}$, the reverse of $\pi$ in $S_n$, is also $k$-convex.  In particular, for all $n > 1$ and $k$, $f_k(n)$.
\end{prop}

\begin{proof}
This is true simply because $\pi(i) = \pi^{R}(n+1-i)$ for $i \in [n]$, and so $\pi(i-1) + \pi(i+1) - 2\pi(i) \leq k \implies \pi^R(n-i+2) + \pi^R(n-i) + 2\pi^R(n-i+1) \leq k$ for $i \in [2, n-1]$.  
\end{proof}

\begin{prop}
Strongly convex permutations (and therefore perfectly convex permutations) avoid consecutive entries order-isomorphic to 213 or 312.
\end{prop}

\begin{proof}
Suppose without loss of generality that, for some $i$, $\pi(i) < \pi(i-1) < \pi(i+1)$ (ie. there exists a consecutive 213), then $\pi(i-1) + \pi(i+1) - 2\pi(i) = (\pi(i+1) - \pi(i)) - (\pi(i) - \pi(i-1)) \leq 2 + 1 = 3$.
\end{proof}

In particular, this gives 0-, 1-, and 2-convex permutations some nice structure.  Notice that another way of phrasing this result is that strongly and perfectly convex permutations contain only substrings whose minima are located at either the beginning or the end of the substring.  Further, we can say that 213 and 312 avoiding permutations consist of an increasing substring followed by a decreasing substring, giving strongly convex permutations a ``mountain'' shape.  Ignoring the convexity condition, it is clear that these patterns are in bijection with 2-colorings of $[n-1]$.  To construct the bijection, color each integer in $[n-1]$ red or blue, then begin with the increasing substring of red integers.  Append $n$, followed by the decreasing substring of blue integers to form a unique permutation. 

\begin{exmp}

\textcolor{blue}{12}\textcolor{red}{3}\textcolor{blue}{45}\textcolor{red}{67}\textcolor{blue}{8} $\to$ \textcolor{red}{367}9\textcolor{blue}{85421}

\textcolor{red}{123}\textcolor{blue}{4}\textcolor{red}{56}\textcolor{blue}{78} $\to$ \textcolor{red}{12356}9\textcolor{blue}{874}
\end{exmp}

This construction ignores convexity, yet all strongly convex permutations may be constructed uniquely in this fashion. Thus we have our first upper bound on the number of strongly convex permutations:

\begin{lem}
For $k = 1, 2, $ \\ \\ \centerline{${f_k(n) < 2^{n-1}.}$}
\end{lem}

(This is a strict inequality as clearly not all ascending-descending permutations are $2$-convex.)  There are other ways to obtain this result.  One requires use of the fact that a strongly convex permutation of length $n$ inherits a strongly convex substring consisting of $[n-1]$.  In other words, if a permutation of length $n-1$ is not strongly convex, there is no way to \enquote{fix} the permutation by inserting $n$.  Let $\pi:[n-1] \to [n-1]$ be a permutation that is not strongly convex, so there is some $i$ for which $\pi(i-1) + \pi(i+1) - 2\pi(i) > 2$.  If one were to attempt to ``fix'' this permutation by extending $\pi$ to $\pi'$, a permutation of length $n$, it is clear that one must insert $n$ next to $\pi(i)$ in the one line notation to fix the convexity condition.  This will not work, however, as replacing $\pi(i-1)$ or $\pi(i+1)$ with $n$ will only worsen the convexity at this point. 

Now, given a strongly convex permutation $\pi(1)\pi(2)\pi(3)\ldots\pi(n)$, how might one build a permutation of length $n+1$?  The ascending-descending nature of the permutation dictates that if $\pi(i) = n,$ it must be that either $\pi(i-1) = n-1$ or $\pi(i+1) = n-1$.  For the same reason, it must be that the only possible positions in which we can place $n+1$ to extend $\pi$ to $\pi'$ and preserve strong convexity are immediately before and immediately after $n$.  Actually, one of these placements is always available; $\pi$ may always be extended to length $n+1$ by placing $n+1$ between $n$ and $n-1$.  Doing this for each permutation, we attain the inequality 
\[
1 \leq f_k(n+1)/f_k(n) \leq 2 \text{ for } k = 1, 2. 
\]
This method of tracking growth preserves substrings by ``building in the middle'' of the permutation.  A similar thing may be done by ``building from the outside.''  We will discuss this method later. Note that the classes of  $1$-convex and $2$-convex permutations are smaller than $Av(213,312)$, which itself is enumerated by $2^{n-1}$, as is known from the seminal paper of Simion and Schmidt \cite{Simion85}.

We now give another result that will prove useful in enumerating strongly convex and perfectly convex permutations.

\begin{prop}
If $\pi$ is $k$-convex with $k < 3$, exactly one of the following must be true: 
\begin{itemize}
\item $\pi(1) = 1$ and $\pi(2) = 2$, or
\item $\pi(n) = 1$ and $\pi(n-1) = 2$, or
\item $\pi(1) = 1$ and $\pi(2) = 3$, or
\item $\pi(n) = 1$ and $\pi(n-1) = 3$, or
\item $\pi(1) = 2$ and $\pi(2) = 3$, or
\item $\pi(n) = 2$ and $\pi(n-1) = 3$.
\end{itemize}
\end{prop} 

\begin{proof}
Since these patterns are 213 and 312 avoiding, we see that either $\pi(1) = 1$ or $\pi(n) = 1$.  We assume now that $\pi(1) = 1$ and recall that since $\pi^{R}$ is also $k$-convex, it remains only to show that $\pi(2) = 2$ or $\pi(2) = 3$ or both $\pi(n) = 2$ and $\pi(n-1) = 3$.  Suppose $\pi(2) > 3$, then since $\pi$ is 213 and 312 avoiding (mountain-shaped, as described earlier), we know the decreasing sequence at the end of the permutation ends with the substring 32.  This proves the proposition. 
\end{proof}

With these facts in mind, enumerating perfectly convex permutations is rather straightforward.

\begin{thm}
For all $n \geq 1$,

\[
f_0(n) = \begin{cases} 1 &\mbox{\emph{if} } n = 1 \\
2 &\mbox{\emph{if} } n = 2 \\
4 &\mbox{\emph{if} } n = 3 \\
6 &\mbox{\emph{if} } n = 4 \\
8 &\mbox{\emph{if} } n \geq 5 \end{cases} 
\]
\end{thm}

\begin{proof}For $n = 1, 2$, this is trivially true, as the set $\{2, \ldots, n-1\}$ is empty.  For $n = 3$, it is easy to check that the only non-perfectly convex permutations are 213 and 312, so $f_0(3) = 3! - 2 = 4$.  Now let $n \geq 4$; we will construct all perfectly convex permutations.  To begin, assume that $\pi(1) = 1$; this will count exactly half of the desired permutations.  Assume now that $\pi(2) = 2$, then we are forced to have $\pi(3) = 3$ to obey the convexity condition, and similarly that $\pi(4) = 4, \ldots, \pi(n) = n$; this is the identity permutation 

\begin{equation} \label{k1}{1234\ldots n} \end{equation} 
Suppose now that $\pi(2) = 3$, then we know that $\pi(n) = 2$ and we must enumerate the ways to fill the middle of the permutation.  If $\pi(3) = 4$, this forces $\pi(4) = 5, \pi(5) = 6, \ldots, \pi(n-1) = n$, so we have the permutation 

\begin{equation} \label{eq:k2}{1345\ldots n2} \end{equation}

The only other choice for $\pi(3)$, given $\pi(1) = 1, \pi(2) = 3$, is 5, which forces $\pi(n-1) = 4$, and then $\pi(n-2) = 6$, $\pi(4) = 7$\ldots so we have 

\begin{equation} \label{eq:k3}1357\ldots n \ldots 8642 \end{equation} 

This covers the first two of the three cases listed above.  In the third situation, insisting that $\pi(1) = 1$, we have $\pi(n) = 2$ and $\pi(n-1) = 3$.  This then forces $\pi(n-2) = 4, \pi(n-3) = 5, \ldots, \pi(3) = n-1, \pi(2) = n$, so we have 

\begin{equation} \label{eq:k4}{1n\ldots 5432} \end{equation} 

We have exhausted all possibilities when requiring that $\pi(1) = 1$.  Notice that for $n = 4$, permutations \ref{eq:k2} and \ref{eq:k3} are the same.  The reverses of these permutations constitute the remaining perfectly convex permutations, hence the result. \end{proof}

Enumerating strongly convex permutations has proven much more difficult than the perfect case.  We now give another method of tracking the growth of these permutations: by considering building ``from the outside.''

\begin{defn}
Let $\pi\in S_n$ be $k$-convex.  Define 
\[
L(\pi) := 1(\pi(1)+1)(\pi(2)+1)\ldots(\pi(n)+1)
\]and 
\[
R(\pi) := (\pi(1)+1)(\pi(2)+1)\ldots(\pi(n)+1)(1)
\]to be, respectively, the \emph{left and right descendants of $\pi$}.    If $L(\pi)$ is $k$-convex, we say that $\pi$ \emph{left descends}, and similarly define what it means for a permutation to \emph{right descend}.

Analogously, let $W = W_1W_2\ldots W_n$ be a word of length $n$ on the alphabet $\{L, R\}$, then $W(\pi)$ is the composition $W_n(W_{n-1}(\ldots(W_2(W_1(\pi)))\ldots))$.  We say that \emph{$\tau$ is a descendant of $\pi$} and \emph{$\pi$ is an ancestor of $\tau$} if there exists some word $W$ on $\{L, R\}$ such that $W(\pi) = \tau$.
\end{defn}

\begin{prop}
Let $k < 3$.  If $\pi$ is a $k$-convex permutation of length $n$ with $\pi(1) = 1$, then $L^{R}(\pi)$ is a $k$-convex permutation of length $n-1$.  Similarly, if $\pi(n) = 1$, then $R^{R}(\pi)$ is a $k$-convex permutation of length $n-1$.
\end{prop}

Hence every $k$-convex permutation beginning with one descends from another $k$-convex permutation.  For example, the 1-convex permutation 123564 descends from 12453, and  its reverse, 465321, descends from 35421, the reverse of the ancestor of the original permutation.  This gives a mapping of strongly convex permutations of length $n$ (which must begin or end with 1) onto the strongly convex permutations of length $n-1$.  Now consider this process in reverse; if $\pi$ is $k$-convex for $k \in \{0, 1, 2\}$ and $\pi(2) - 2\pi(1) \leq k$, then $L(\pi)$ is $k$-convex ($\pi$ left descends), and similarly if $\pi(n-1) - 2\pi(n) \leq k$, then $R(\pi)$ is $k$-convex ($\pi$ right descends).  Notice that for $k = $ 1 or 2, every $k$-convex permutation $\pi$ has either a $k$-convex right descendant or a $k$-convex left descendant, as we have shown previously that $\pi$ must begin with 12, 13, or 23 or end with 21, 31, or 32, all of which satisfy the condition for a descendant to be $k$-convex.

\begin{figure}[!ht]
The graph displayed here is half of the first seven generations of the full 1-convex tree showing descendants of 12; the other half begins with 21 and branches in the same manner except that each descendant is placed by its reverse. \\

\centerline{\includegraphics[height = 4.2in]{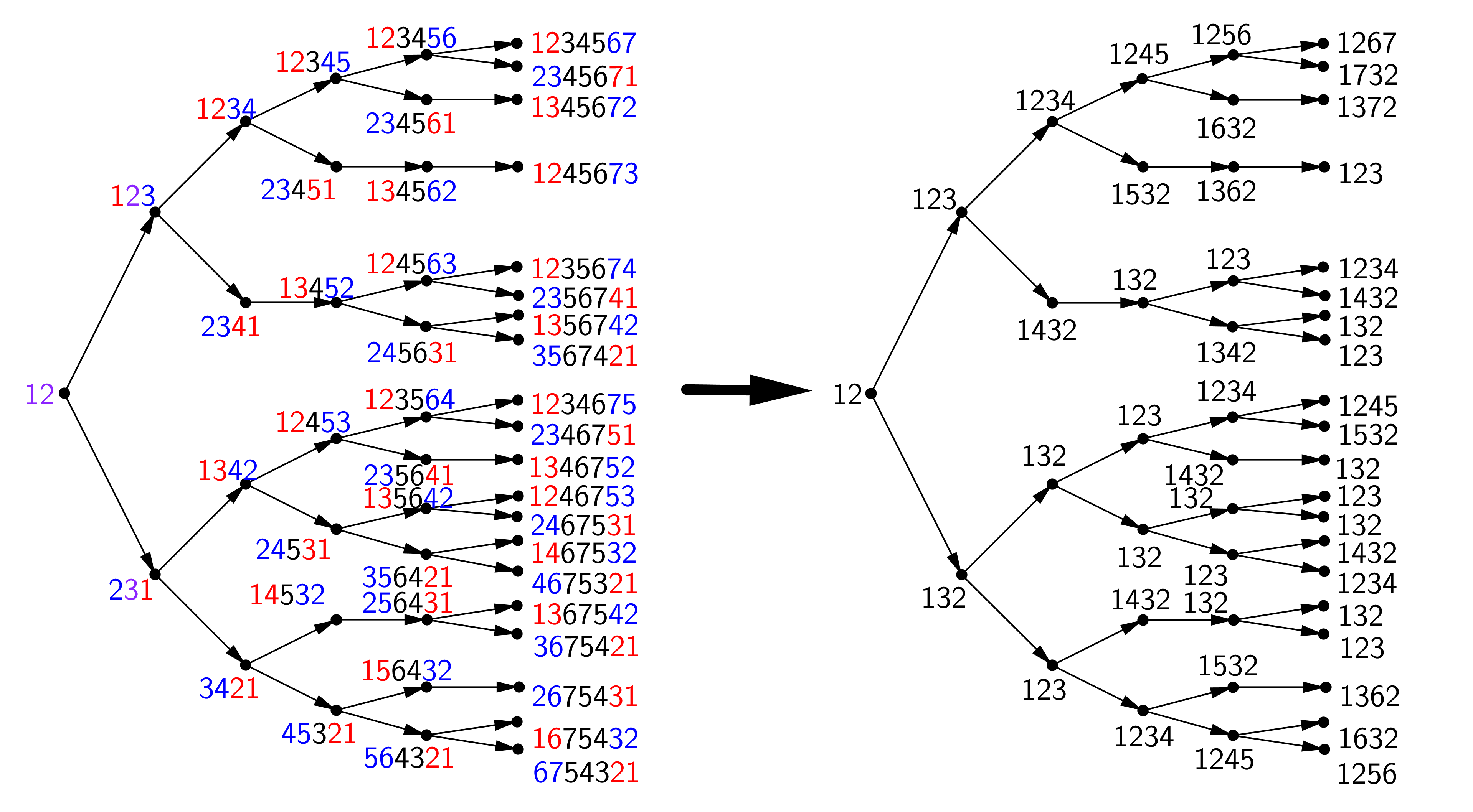}}
\begin{center}
1-convex descendant graph (full permutation and abbreviated by endpoints, seven generations)
\end{center}
\end{figure}


Using this method to count and track strongly convex permutations, we determine $f_k(n)$ for $k \leq 2, n \leq 12$.

\begin{table}[!ht]
\caption{Number of Strongly Convex Permutations for Small $n$}
\centering
\begin{tabular}{c c c c}
\hline
$n$ & $f_0(n)$ & $f_1(n)$ & $f_2(n)$ \\
\hline
1 & 1 & 1 & 1 \\
2 & 2 & 2 & 2 \\
3 & 4 & 4 & 4 \\
4 & 6 & 8 & 8 \\
5 & 8 & 14 & 16 \\
6 & 8 & 24 & 30 \\
7 & 8 & 40 & 56 \\
8 & 8 & 66 & 102 \\
9 & 8 & 106 & 186 \\
10 & 8 & 170 & 336 \\
11 & 8 & 270 & 606 \\
12 & 8 & 426 & 1088
\end{tabular}
\end{table}

\begin{defn}
Let $\pi$ be a $k$-convex permutation of length $n$.  Define 
\[
d_k(\pi, i) :=  |k\text{-convex descendants of }\pi \text{ of length }n+i| = |\{W \in \{L, R\}^i \,|\, W(\pi) \text{ is $k$-convex}\}|
\]

We say that two permutations $\pi$ and $\tau$ are \emph{identically-descending with respect to $k$} if $d_k(\pi, i) = d_k(\tau, i)$ for all $i \in \mathbb{N}$.  In this case, we write $d_k(\pi) = d_k(\tau)$.
\end{defn}

Notice that by this definition, every permutation is identically-descending to its reverse, as the reverse permutation of $W(\pi)$ is $W^C(\pi^{R})$, where $W'$ is formed from $W$ by changing all $L$'s to $R$'s and all $R$'s to $L$'s.  We generalize this in the following Lemma.

\begin{lem}
Let $\pi$ be a permutation of length $n$ and let $\tau$ be a permutation of length $m$.  If $\pi(1) = \tau(1), \pi(2) = \tau(2)$, $\pi(n-1) = \tau(m-1)$, and $\pi(n) = \tau(m)$, or, if $\pi(1) = \tau(m), \pi(2) = \tau(m-1), \pi(n-1) = \tau(2)$, and $\pi(n) = \tau(1)$, then $d_k(\pi) = d_k(\tau)$.
\end{lem}

This is true simply because when considering whether the left and right descendants of a permutation are strongly convex, the only new condition to check is at the left or right endpoint, respectively.  This suggests that when considering strongly convex permutations, we can abbreviate the permutations $1a\ldots bc$ and $cb\ldots a1$ by $1abd$ and not lose any information about its descendants.  This result allows us to simplify the graph by removing some of the vertices, but it is possible to do even better:

\begin{thm} \label{thm:id}
Fix $k \in \{0, 1, 2\}$ and let $\pi = ab\ldots cd$ and $\tau = ab\ldots c'd$ be two $k$-convex permutations.  If $R(\pi)$ and $R(\tau)$ are $k$-convex, then $d_k(\pi) = d_k(\tau)$.
\end{thm}

\begin{proof}Let $W_\pi$ be the set of words on the alphabet $\{L, R\}$ such that $W(\pi)$ is $k$-convex for every $W \in W_\pi$; similarly define $W_\tau$.  Suppose $R\sigma \in W_\pi$ for some word $\sigma$, then $R\sigma(\tau) = \sigma(R(\tau)) = \sigma((a+1)(b+1)\ldots(c+1)(d+1)1)$.  The permutation $(a+1)(b+1)\ldots(c+1)(d+1)1$ shares the four-letter abbreviation $1(d+1)(b+1)(a+1)$ with $(a+1)(b+1)\ldots(c'+1)(d+1)1) = R(\tau)$, so it must be that $R\sigma(\tau)$ is $k$-convex, and so $R(\sigma) \in W_\tau$.  Analogously, $R\sigma \in W_\tau \implies R\sigma \in W_\pi$.

For any $n \in \mathbb{N}$, clearly $L^n (= LL\ldots L) \in W_\pi \iff L^n \in W_\tau$ since the convexity of the left descendants of $\pi$ and $\tau$ depend only on $a$ and $b$, which are shared.  Finally, for any $n \in \mathbb{N}$ and any $\sigma$ on $\{L, R\}$, \\ \\
\centerline{$L^nR\sigma \in W_\pi \iff L^n \in W_\pi$ and $R\sigma \in W_{L^n(\pi)}$.} \\

Notice that $L^n(\pi) = 12\ldots n (a+n)(b+n) \ldots (c+n)(d+n)$ and $L^n(\tau) = 12\ldots n (a+n)(b+n) \ldots (c'+n)(d+n)$ can be abbreviated $12ef$ and $12e'f$ for $e = c+n, e'  c' + n, f = d+n$, and $12ef$ and $12e'f$ both right descend to $23(f+1)1$ since $e - 2f = c+n - 2d - 2n < c - 2d \leq k, e' - 2f = e'+n - 2d - 2n < e' - 2d \leq k$.  Thus \\ \\
\centerline{$L^n \in W_\pi$ and $R\sigma \in W_{L^n(\pi)} \iff L^n \in W_\tau$ and $R\sigma \in W_{L^n(\tau)} \iff L^nR\sigma \in W_\tau.$} \\

This concludes the proof that $W_\pi = W_\tau$, and so $\pi$ and $\tau$ are identically-descending.
\end{proof}

We may now express many permutations by the same four-character abbreviation without losing any information about their descendants, choosing the smallest numbers possible for convenience.  For example, whereas we previous abbreviated the permutations such as 123564 and 46785321 by 1264, we may now abbreviate these permutations, along with others, including 12354, 487654321, and 1235674, as 1234; all of these 2-convex permutations are identically-descending with respect to two.  The ability to express the behavior of multiple permutations using the same abbreviation suggests that recursion may be an effective tool in attempting to enumerate.  We previously displayed convex permutations as a directed tree with edges from permutations to their first descendants.  We now display the diagram again, this time allowing for cycles and reducing each permutation to the ``smallest'' abbreviation possible:

\begin{figure}[!ht]
\centerline{\includegraphics[height = 3.8in]{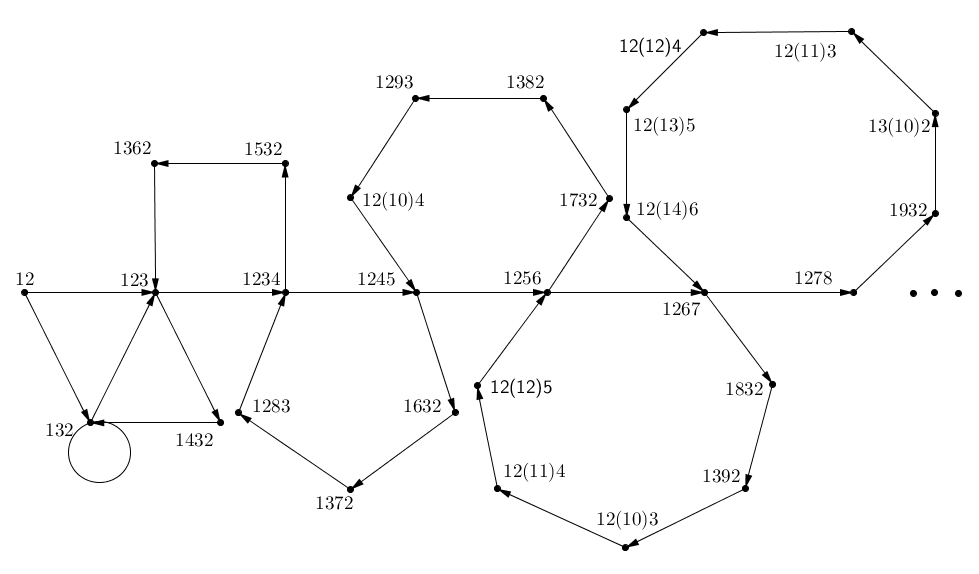}}
\begin{center}
1-convex identically-descending transition digraph
\end{center}
\end{figure}

\setcounter{MaxMatrixCols}{24}
$A = \begin{bmatrix} 
0 & 1 & 0 & 1 & 0 & 0 & 0 & 0 & 0 & 0 & 0 & 0 & 0 & 0 & 0 & 0 & 0 & 0 & 0 & 0 & 0 & 0 & \ldots \\
0 & 1 & 0 & 1 & 0 & 0 & 0 & 0 & 0 & 0 & 0 & 0 & 0 & 0 & 0 & 0 & 0 & 0 & 0 & 0 & 0 & 0 & \ldots \\
0 & 1 & 0 & 0 & 0 & 0 & 0 & 0 & 0 & 0 & 0 & 0 & 0 & 0 & 0 & 0 & 0 & 0 & 0 & 0 & 0 & 0 & \ldots \\
0 & 0 & 1 & 0 & 1 & 0 & 0 & 0 & 0 & 0 & 0 & 0 & 0 & 0 & 0 & 0 & 0 & 0 & 0 & 0 & 0 & 0 & \ldots \\
0 & 0 & 0 & 0 & 0 & 1 & 0 & 1 & 0 & 0 & 0 & 0 & 0 & 0 & 0 & 0 & 0 & 0 & 0 & 0 & 0 & 0 & \ldots \\
0 & 0 & 0 & 0 & 0 & 0 & 1 & 0 & 0 & 0 & 0 & 0 & 0 & 0 & 0 & 0 & 0 & 0 & 0 & 0 & 0 & 0 & \ldots \\
0 & 0 & 0 & 1 & 0 & 0 & 0 & 0 & 0 & 0 & 0 & 0 & 0 & 0 & 0 & 0 & 0 & 0 & 0 & 0 & 0 & 0 & \ldots \\
0 & 0 & 0 & 0 & 0 & 0 & 0 & 0 & 1 & 0 & 0 & 1 & 0 & 0 & 0 & 0 & 0 & 0 & 0 & 0 & 0 & 0 & \ldots \\
0 & 0 & 0 & 0 & 0 & 0 & 0 & 0 & 0 & 1 & 0 & 0 & 0 & 0 & 0 & 0 & 0 & 0 & 0 & 0 & 0 & 0 & \ldots \\
0 & 0 & 0 & 0 & 0 & 0 & 0 & 0 & 0 & 0 & 1 & 0 & 0 & 0 & 0 & 0 & 0 & 0 & 0 & 0 & 0 & 0 & \ldots \\
0 & 0 & 0 & 0 & 1 & 0 & 0 & 0 & 0 & 0 & 0 & 0 & 0 & 0 & 0 & 0 & 0 & 0 & 0 & 0 & 0 & 0 & \ldots \\
0 & 0 & 0 & 0 & 0 & 0 & 0 & 0 & 0 & 0 & 0 & 0 & 1 & 0 & 0 & 0 & 1 & 0 & 0 & 0 & 0 & 0 & \ldots \\
0 & 0 & 0 & 0 & 0 & 0 & 0 & 0 & 0 & 0 & 0 & 0 & 0 & 1 & 0 & 0 & 0 & 0 & 0 & 0 & 0 & 0 & \ldots \\
0 & 0 & 0 & 0 & 0 & 0 & 0 & 0 & 0 & 0 & 0 & 0 & 0 & 0 & 1 & 0 & 0 & 0 & 0 & 0 & 0 & 0 & \ldots \\
0 & 0 & 0 & 0 & 0 & 0 & 0 & 0 & 0 & 0 & 0 & 0 & 0 & 0 & 0 & 1 & 0 & 0 & 0 & 0 & 0 & 0 & \ldots \\
0 & 0 & 0 & 0 & 0 & 0 & 0 & 1 & 0 & 0 & 0 & 0 & 0 & 0 & 0 & 0 & 0 & 0 & 0 & 0 & 0 & 0 & \ldots \\
0 & 0 & 0 & 0 & 0 & 0 & 0 & 0 & 0 & 0 & 0 & 0 & 0 & 0 & 0 & 0 & 0 & 1 & 0 & 0 & 0 & 0 & \ldots \\
0 & 0 & 0 & 0 & 0 & 0 & 0 & 0 & 0 & 0 & 0 & 0 & 0 & 0 & 0 & 0 & 0 & 0 & 1 & 0 & 0 & 0 & \ldots \\
0 & 0 & 0 & 0 & 0 & 0 & 0 & 0 & 0 & 0 & 0 & 0 & 0 & 0 & 0 & 0 & 0 & 0 & 0 & 1 & 0 & 0 & \ldots \\
0 & 0 & 0 & 0 & 0 & 0 & 0 & 0 & 0 & 0 & 0 & 0 & 0 & 0 & 0 & 0 & 0 & 0 & 0 & 0 & 1 & 0 & \ldots \\
0 & 0 & 0 & 0 & 0 & 0 & 0 & 0 & 0 & 0 & 0 & 0 & 0 & 0 & 0 & 0 & 0 & 0 & 0 & 0 & 0 & 1 & \ldots \\
0 & 0 & 0 & 0 & 0 & 0 & 0 & 0 & 0 & 0 & 0 & 1 & 0 & 0 & 0 & 0 & 0 & 0 & 0 & 0 & 0 & 0 & \ldots \\
\vdots & \vdots & \vdots & \vdots & \vdots & \vdots & \vdots & \vdots & \vdots & \vdots & \vdots & \vdots & \vdots & \vdots & \vdots & \vdots & \vdots & \vdots & \vdots & \vdots & \vdots & \vdots & \ddots
\end{bmatrix}$

Let $A$ be the (infinite) adjacency matrix for the 1-convex descendant graph.  We know then that since $A^n_{ij}$ is the number of walks of length $n$ from the $i^{\text{th}}$ vertex to the $j^{\text{th}}$ and a 1-convex permutation of length $n+1$ is determined uniquely by a walk of length $n$ beginning at the vertex 12, we have 
\[
\sum_{j = 1}^\infty A^{n-1}_{1j} = \frac{1}{2}f_1(n).
\]

The issue is that because the graph is infinite, so is the matrix.  We can obtain a lower bound on the growth rate of the sequence by truncating the graph at the edge from 1267 to 1278 (and therefore the corresponding adjacency matrix) and calculating $(I - Ax)^{-1}$, which is formally equal to $\sum_{n=0}^\infty A^nx^n$.  Each cell in the resulting matrix will contain a generating function for the corresponding cell of $A$.  As we are interested in the number of walks beginning at the vertex 12, represented by the first row and column of the adjacency matrix, we sum over the first row to obtain a generating function for half the number of walks of length $n+1$.  Using the first 22 rows and columns of the matrix (as pictured above), we obtain: 
\[
F_{1, -}(x) = \frac{-1 - x - 2 x^2 - 2 x^3 - 2 x^4 - 2 x^5 - x^6 + x^7 + x^8 + 
 2 x^9 + x^{10} + x^{11} - x^{14}}{-1 + x + x^3 + x^4 - 2 x^8 - x^9 - 2 x^{10} + x^{13} + x^{15}}
\]

We now have a rational generating function whose denominator has a unique minimal root $r = 0.65149869151455837735\ldots$, which, as described in \cite[p. 171]{Wilf94} gives us an asymptotic lower bound: 

\[
f_1(n) > f_{1, -}(n) = \Theta\left(\frac{1}{r}^n\right) = \Theta(1.5349224995\ldots^n)
\]

To find an upper bound, consider the edge of the graph that goes from 1267 to 1278.  Because clearly 1267 has more descendants than 1278, we can truncate the graph at 1267 and insert a loop from 1267 to itself.  The adjacency matrix for this upper bound graph is the same as the one above, except that we must insert a 1 into the $22^{\text{nd}}$ row and column.  Again, we obtain a rational generating function: 
\[
F_{1, +}(x) = \frac{-1 - x^2 + x^6 + 2 x^7 + x^8 + 2 x^9 + x^{10} + x^{11} - x^{14}}{-1 + 2 x - x^2 + x^3 - x^5 - x^8 - x^{10} + x^{11} - x^{12} + x^{13} + x^{15}}
\]

$F_{1, +}$ also has a unique minimal root of $s = 0.65145978572056851317\ldots$, which gives the asymptotic upper bound: 
\[
f_1(n) < f_{1, +}(n) =  \Theta\left(\frac{1}{s}^n\right) =\Theta(1.535014167\ldots^n)
\] 

We can use the same procedure to estimate the asymptotic growth of $f_2(n)$.  The digraph used in this case is slightly different: \\

\begin{figure}[!ht]
\centerline{\includegraphics[height = 4.2in]{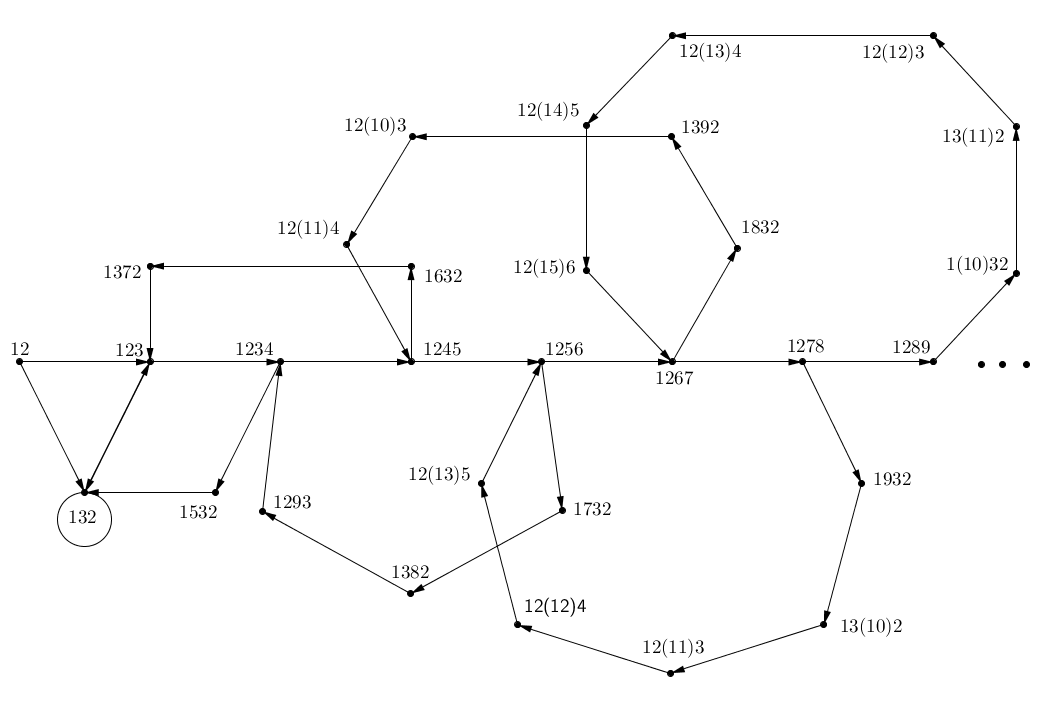}}
\begin{center}
2-convex identically-descending transition digraph
\end{center}
\end{figure}

\setcounter{MaxMatrixCols}{25}
$B = \begin{bmatrix} 
0 & 1 & 1 & 0 & 0 & 0 & 0 & 0 & 0 & 0 & 0 & 0 & 0 & 0 & 0 & 0 & 0 & 0 & 0 & 0 & 0 & 0 & 0 & 0 & 0 \ldots \\
0 & 0 & 1 & 1 & 0 & 0 & 0 & 0 & 0 & 0 & 0 & 0 & 0 & 0 & 0 & 0 & 0 & 0 & 0 & 0 & 0 & 0 & 0 & 0 & 0 \ldots \\
0 & 1 & 1 & 0 & 0 & 0 & 0 & 0 & 0 & 0 & 0 & 0 & 0 & 0 & 0 & 0 & 0 & 0 & 0 & 0 & 0 & 0 & 0 & 0 & 0 \ldots \\
0 & 0 & 0 & 0 & 1 & 0 & 1 & 0 & 0 & 0 & 0 & 0 & 0 & 0 & 0 & 0 & 0 & 0 & 0 & 0 & 0 & 0 & 0 & 0 & 0 \ldots \\
0 & 0 & 0 & 0 & 0 & 1 & 0 & 0 & 0 & 0 & 0 & 0 & 0 & 0 & 0 & 0 & 0 & 0 & 0 & 0 & 0 & 0 & 0 & 0 & 0 \ldots \\
0 & 1 & 1 & 0 & 0 & 0 & 0 & 0 & 0 & 0 & 0 & 0 & 0 & 0 & 0 & 0 & 0 & 0 & 0 & 0 & 0 & 0 & 0 & 0 & 0 \ldots \\
0 & 0 & 0 & 0 & 0 & 0 & 0 & 1 & 0 & 1 & 0 & 0 & 0 & 0 & 0 & 0 & 0 & 0 & 0 & 0 & 0 & 0 & 0 & 0 & 0 \ldots \\
0 & 0 & 0 & 0 & 0 & 0 & 0 & 0 & 1 & 0 & 0 & 0 & 0 & 0 & 0 & 0 & 0 & 0 & 0 & 0 & 0 & 0 & 0 & 0 & 0 \ldots \\
0 & 1 & 0 & 0 & 0 & 0 & 0 & 0 & 0 & 0 & 0 & 0 & 0 & 0 & 0 & 0 & 0 & 0 & 0 & 0 & 0 & 0 & 0 & 0 & 0 \ldots \\
0 & 0 & 0 & 0 & 0 & 0 & 0 & 0 & 0 & 0 & 1 & 0 & 0 & 1 & 0 & 0 & 0 & 0 & 0 & 0 & 0 & 0 & 0 & 0 & 0 \ldots \\
0 & 0 & 0 & 0 & 0 & 0 & 0 & 0 & 0 & 0 & 0 & 1 & 0 & 0 & 0 & 0 & 0 & 0 & 0 & 0 & 0 & 0 & 0 & 0 & 0 \ldots \\
0 & 0 & 0 & 0 & 0 & 0 & 0 & 0 & 0 & 0 & 0 & 0 & 1 & 0 & 0 & 0 & 0 & 0 & 0 & 0 & 0 & 0 & 0 & 0 & 0 \ldots \\
0 & 0 & 0 & 1 & 0 & 0 & 0 & 0 & 0 & 0 & 0 & 0 & 0 & 0 & 0 & 0 & 0 & 0 & 0 & 0 & 0 & 0 & 0 & 0 & 0 \ldots \\
0 & 0 & 0 & 0 & 0 & 0 & 0 & 0 & 0 & 0 & 0 & 0 & 0 & 0 & 1 & 0 & 0 & 0 & 1 & 0 & 0 & 0 & 0 & 0 & 0 \ldots \\
0 & 0 & 0 & 0 & 0 & 0 & 0 & 0 & 0 & 0 & 0 & 0 & 0 & 0 & 0 & 1 & 0 & 0 & 0 & 0 & 0 & 0 & 0 & 0 & 0 \ldots \\
0 & 0 & 0 & 0 & 0 & 0 & 0 & 0 & 0 & 0 & 0 & 0 & 0 & 0 & 0 & 0 & 1 & 0 & 0 & 0 & 0 & 0 & 0 & 0 & 0 \ldots \\
0 & 0 & 0 & 0 & 0 & 0 & 0 & 0 & 0 & 0 & 0 & 0 & 0 & 0 & 0 & 0 & 0 & 1 & 0 & 0 & 0 & 0 & 0 & 0 & 0 \ldots \\
0 & 0 & 0 & 0 & 0 & 0 & 1 & 0 & 0 & 0 & 0 & 0 & 0 & 0 & 0 & 0 & 0 & 0 & 0 & 0 & 0 & 0 & 0 & 0 & 0 \ldots \\
0 & 0 & 0 & 0 & 0 & 0 & 0 & 0 & 0 & 0 & 0 & 0 & 0 & 0 & 0 & 0 & 0 & 0 & 0 & 1 & 0 & 0 & 0 & 0 & 0 \ldots \\
0 & 0 & 0 & 0 & 0 & 0 & 0 & 0 & 0 & 0 & 0 & 0 & 0 & 0 & 0 & 0 & 0 & 0 & 0 & 0 & 1 & 0 & 0 & 0 & 0 \ldots \\
0 & 0 & 0 & 0 & 0 & 0 & 0 & 0 & 0 & 0 & 0 & 0 & 0 & 0 & 0 & 0 & 0 & 0 & 0 & 0 & 0 & 1 & 0 & 0 & 0 \ldots \\
0 & 0 & 0 & 0 & 0 & 0 & 0 & 0 & 0 & 0 & 0 & 0 & 0 & 0 & 0 & 0 & 0 & 0 & 0 & 0 & 0 & 0 & 1 & 0 & 0 \ldots \\
0 & 0 & 0 & 0 & 0 & 0 & 0 & 0 & 0 & 0 & 0 & 0 & 0 & 0 & 0 & 0 & 0 & 0 & 0 & 0 & 0 & 0 & 0 & 1 & 0 \ldots \\
0 & 0 & 0 & 0 & 0 & 0 & 0 & 0 & 0 & 1 & 0 & 0 & 0 & 0 & 0 & 0 & 0 & 0 & 0 & 0 & 0 & 0 & 0 & 0 & 0 \ldots \\
\vdots & \vdots & \vdots & \vdots & \vdots & \vdots & \vdots & \vdots & \vdots & \vdots & \vdots & \vdots & \vdots & \vdots & \vdots & \vdots & \vdots & \vdots & \vdots & \vdots & \vdots & \vdots & \vdots & \vdots & \ddots
\end{bmatrix}$

Similarly, here we will find a lower bound on the growth rate by truncating the graph at the edge between 1278 and 1289, and find an upper bound by creating a loop from 1278 to itself, using the resulting matrices to find generating functions for the cells in the first row:

\[
F_{2, -}(x) = \frac{1 + x + x^2 + 2 x^3 + 2 x^4 + 2 x^5 + 2 x^6 + x^7 - x^{10} - x^{11} - x^{12}}{1 - x - x^2 - x^4 - x^5 + x^8 + 
 2 x^9 + x^{10} + x^{11} + x^{12} + x^{13} - x^{14}}
\]

$F_{2, -}$ also has a unique minimal root of $t = 0.55979335021175578170\ldots$, which gives the asymptotic upper bound: 
\[
f_2(n) > f_{2, -}(n) =  \Theta\left(\frac{1}{t}^n\right) =\Theta(1.786373489\ldots^n)
\] 

And, 
\[
F_{2, -}(x) = \frac{1 + x^3 - x^7 - x^8 - x^9 - 2 x^{10} - x^{11} - x^{12}}{1 - 2x + x^3 - x^4 + x^6 + x^8 + x^{11} + x^{13} - x^{14}}
\]

$F_{1, +}$ also has a unique minimal root of $u = 0.55977426822528580510\ldots$, which gives the asymptotic upper bound: 
\[
f_2(n) < f_{2, +}(n) =  \Theta\left(\frac{1}{u}^n\right) =\Theta(1.786434384\ldots^n)
\]

It certainly appears that $f_1(n)$ and $f_2(n)$ grow by some asymptotic ratios $C_1$ and $C_2$, which are around 1.535 and 1.786, respectively, and that taking larger and larger transfer matrices will provide us with better and better approximations of these values.  It turns out that this is the case, formalized as follows:

\begin{cor}
For $k=1$ and $k=2$, $f_k(n)$ is $\Theta(C_k^n)$ for some finite $C_k$.
\end{cor}

\begin{proof}  
We proceed by using subadditivity techniques, done as in \cite{Arratia99}.  We have that for $k=1,2, f_k(m+n)\le f_k(m)f_k(n)$ since a single-peaked $k$-convex permutation $\pi$  on $[m+n]$ may be uniquely separated out into single-peaked k-convex permutations on $[m]$ and $[n]$. This may be done by first taking the subsequence of $\pi$ consisting of the numbers $[m]$, call this $\tau$. To find the second sequence, we then take the remaining $n$ entries of $\pi$ and subtract $m$ from each entry, call this $\rho$. For example: $1246753$ may be separated for $m=3$ and $n=4$ into $\tau=123$ and $\rho=4675$.  It is easy to see that $\rho$ will always be a $k$-convex permutation of length $n$ by its inclusion in $\pi$. So, we need only check that $\tau$ itself is actually such a permutation. However, this is clear, as we have seen that deleting the peak of a $k$-convex permutation gives a $k$-convex permutation of a smaller length. Thus we may delete each entry of $\rho$ from $\pi$ in decreasing order to reach $\tau$, at which point we see that it is a $k$-convex permutation of length $[m]$.

Taking logarithms, this shows that $$\log f_k(m+n)\le \log f_k(m)+\log f_k(n),$$
and so, by Fekete's lemma we have that
$$\lim \frac{\log f_k(n)}{n}=L<\infty,$$
which gives us the fact (since the exponential function is continuous) that
$\lim f_k^{1/n}(n)=C_k$ exists, as asserted. This then gives that $f_k(n)=\Theta(C_k^n)$. 

\end{proof}

In fact, for any $k\ge 3$, we see that 
$\pi$ is $k$-convex only if it avoids, consecutively, a set $A_k$ of permutations. It is an open question as to what the minimum cardinality of $A_k$ is. Already, we have seen that $\vert A_1\vert=\vert A_2\vert=2$.  It would be interesting to determine whether our local $k$-convexity condition is strong enough to allow for us to make the following Stanley-Wilf conjecture type claim:

\begin{con}\label{conj}
 For each $k > 2$, $\lim_{n \to \infty} (f_k(n))^{\frac{1}{n}} = C_k$ where each $C_k$ is a constant.
\end{con}

If it can be shown that our local property implies that a $k$-convex permutation avoids a non-barred permutation class, then we could apply the Marcus-Tardos theorem \cite{Marcus04} to conclude that $f_k(n)\le C_k^n$ for some $1< C_k<\infty$.

\section{A Generating Function for \($k=1$\)}

In this section, we  find the generating function 
\[
F_1(q) := \sum_{n=0}^\infty f_1(n)q^n.
\] 
Note that while we can show that the limit for the growth constant exists and we can produce arbitrarily tight bounds for the limit by using larger matrices, we are still unable to enumerate these permutations exactly. We will simplify a connected piece of the transition digraph, find sufficient information about that piece in order to analyze it, and then use that information in the transition matrix to finish our analysis. 

From the 1-convex transition digraph, we see that the nodes that are descended from successive left descendants of $123$ are essentially isolated along a path in the graph. We refer to the node $123$ as $1223$ throughout this section as it simplifies our notation. 

To see this, consider a permutation $\sigma$ of the form $12(l)(l+1)$, where $l\ge 3$. Note that if this permutation right descends, then the resulting permutation, $R(\sigma) = 23(l+2)1$, will not right descend for $l-1$ generations, and after successive left descents we have $L^{l-1}R(\sigma) = 12(l+2+(l-1))l$.  This permutation left descends and right descends and thus by Theorem \ref{thm:id} is identically-descending with $12(l-1)l$. Note that as $l\ge 3$, the first two entries of the successive left descents of $R(\sigma)$ eventually take the form $12$.

Using this information, we consider the vertex induced subgraph, $T$, induced in the transfer digraph for $k=1$ by the nodes $\{12(k)(k+1):k\ge 2\}$. We adopt the convention that an edge labeled with a natural number, $k$, represents a path on $k$ edges from the tail to head of the labeled edge. By the above observations, we see that $T$ has the form illustrated in the figure.

\begin{figure}[!t]\label{topk1}
\centering
\begin{tikzpicture}[->,>=stealth',shorten >=1pt,auto,node distance=3cm,
  thick,main node/.style={circle,fill=blue!20,draw,font=\sffamily\Large\bfseries},notmain node/.style={circle,fill=white!20,draw,font=\sffamily\Large\bfseries}]

  \node[main node] (1) {1223};
  \node[main node] (2) [right of=1] {1234};
  \node[main node] (3) [right of=2] {1245};
  \node[main node] (4) [right of=3] {1256};
  \node[notmain node] (5) [right of=4] {...};

  \path[every node/.style={font=\sffamily\small}]
    (1) edge [bend left] node[right] {} (2)
    (2) edge node [bend left] {3} (1)
	    edge [bend left] node [right] {} (3)
    (3) edge node [bend right] {4} (2)
        edge [bend left] node[right] {} (4)
    (4) edge node [bend right] {5} (3)
        edge [bend left] node[right] {} (5)
    (5) edge node [bend right] {} (4);
   
\end{tikzpicture}
\caption{An illustration of the root nodes of $T$.}
\label{T}
\end{figure}
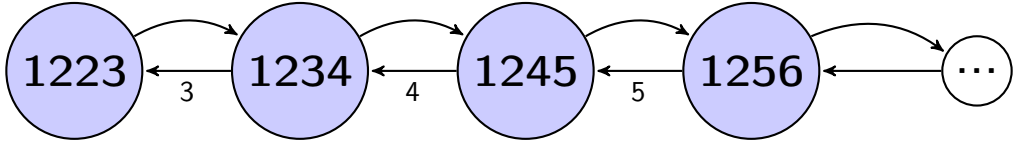

We now wish to count walks in $T$ starting from the node $1234$ and ending at $1234$.

\begin{prop}\label{k1bot}
Let $bot(q)$ be the generating function counting the number of walks in $T$ from $1223$ to $1223$ by length.  Then,

\begin{equation}\label{bot}
bot(q)=\cfrac{1}{1-\cfrac{q^4}{1-\cfrac{q^5}{1-\cfrac{q^6}{\ddots}}}}
\end{equation}
\end{prop}
\begin{proof}
Let $T_k$ be the vertex induced subgraph of $G$ induced by the vertices $\{12(l+1)(l+2):l\ge k\}$. Now let $H_k(q)$ be the generating function for the number of walks from $12(k+2)(k+3)$ to itself counted by length. So, we see that $H_1(q)=bot(q)$. 

Now, note that any walk from $1223$ to itself is either the empty walk, or its first step is from $1223$ to $1234$ and its last three steps return the walk from $1234$ to $1223$. Between those steps, the walk consists of a walk beginning at $1234$ and returning to $1234$, we know these walks are counted by $H_2(q)$. So, from the above considerations we have:
\begin{equation}\label{Trec}
H_1(q)=1+q^4H_1(q)H_2(q).
\end{equation}
Similarly, by starting our walk in $T_k$ and performing the same argument with slight adjustment for the changed path length we see that
\[
H_k(q)=1+q^{3+k}H_k(q)H_{k+1}(q)
\]
Then by repeatedly substituting the values we found in \ref{Trec}, we are able to solve for \ref{bot}.
\end{proof}

The approach used here to find the generating function is similar to that introduced by Flajolet in \cite{Flajolet80} as we can associate each walk in our graph with a word in the alphabet $a_k,b_k$ described in the paper and then evaluate $a_k$ as $q$ and $b_k$ as $q^{(k+2)}$. More generally, we see that this follows by the choice to view any Motzkin path with particular weightings of steps as a walk in a digraph. Further, note that this generating function is quite similar to the generating function for the number of fountains, which has been studied in \cite{Glasser87} and \cite{Odlyzko88}.

We also wish to find the generating function for the number of walks in $T$ beginning at $1223$ counted by length, $tot(q)$, as we will need this later.  

\begin{lem}
\begin{equation}\label{total}
tot(q)=H_1(q)+qH_1(q)H_2(q)(1+q+q^2)+q^2H_1(q)H_2(q)H_3(q)(1+q+q^2+q^3)+...
\end{equation}
or
\[
tot(q)=\sum_{n\ge 0}q^n\frac{1-q^{n+1}}{1-q}\prod_{i=1}^{n+1}H_i(q).
\]
\end{lem}
\begin{proof}
The proof of the above is similar to that of Proposition \ref{k1bot}. Consider the number of walks that end at the first vertex, $1223$; there are $H_1(q)$ of these. Now consider walks ending at $1234$; each of these is an extension of a walk ending at $1223$, and it takes $1$ step to reach this vertex, so there are $qH_1(q)H_2(q)$ such walks. Next, we note that the number of walks ending at each point along the path back to $1223$ from $1234$ is then $qH_1(q)H_2(q)(1+q+q^2)$. We can repeat these observations for each node in the graph and thus arrive at the generating function above. 
\end{proof}

We now will use the transfer matrix method to find the generating function. To do so, we will integrate what we have learned about $T$ into the transfer matrix framework. Note the following: whenever we consider a walk in the graph that reaches the node $1223$, the walk will then branch into other, longer walks staying in $T$ or returning to $1223$ before re-entering the bottom of the graph. By ``bottom'' we mean all vertices not descended from a left descendant of 1223, or in other words, the vertices not in $T$. Instead of keeping track of walks in the upper part of the graph individually, we can use a generating function. To do this consider a walk of length $n$ that ends on $1223$; this corresponds to a permutation of length $n+1$. We now wish to determine the number of successive descendants of this permutation, or equivalently the number of walks in the graph beginning this way. The walk will either leave $1223$ immediately or continue into $T$. If the walk continues into $T$, then we must keep track of how many walks exist in $T$ after $k$ steps. If it were impossible to leave $T$, we would be done and could simply keep track of how many walks there are in $T$ starting at $1223$. However, a walk in $T$ may return to $1223$ after some amount of time, at which point the walk branches and leaves $T$, as well as continuing back up into $T$. If we know how many walks of $l$ steps in $T$ begin and end at $1223$, we can then anticipate these walks leaving $T$ after a certain number of steps. We already have generating functions for both of these quantities, $tot(q)$ and $bot(q)$; $tot(q)$ counts how many walks of length $n$ stay in $T$ if one walk ends at $1223$, and $bot(q)$ counts how many walks in $T$  beginning at $1223$ return to $1223$ after so many steps. Thus we see that $tot(q)-bot(q)$ has for its $n$th coefficient the number of walks in $T$ of length $n$ beginning at $1223$ but not ending at $1223$.

To incorporate this into our transfer matrix framework, we must count both the number of walks in $T$ that will begin at $1223$ after so many steps in addition to the number of walks that do not begin at 1223 but remain entirely in $T$. For the number of walks that do begin at 1223, we can use the generating function $bot(q)$, as this counts the number of walks that begin and end at $1234$ and are entirely in $T$. As $bot(q)$ counts walks by length, we weight the edge between $1223$ and $1432$ by $bot(q)$ instead of $1$ in our transfer matrix. Each time a walk reaches $1223$, we will anticipate the future walks that begin at $1223$, but because those walks are $n$ steps in the future, they carry a coefficient of $q^n$. Similarly, we wish to count those walks that stay entirely in $T$. We noted before that $tot(q)-bot(q)$ counts these walks, and since they have been counted previously these permutations will not have any descendants in our transition digraph. Instead of having an edge from $1223$ to $1234$ in our transition digraph, we will have an edge to a new node $sink$, which has no outgoing edges. We will weight the edge to $sink$ by $tot(q)-bot(q)$ and thus have $tot(q)-bot(q)$ as the corresponding entry in our transfer matrix. Whenever a walk reaches $1223$, exactly two edges will be traversed. One edge returns into the digraph with a weighting reflecting future walks that would be coming from $T$. The other edge goes to a sink and thus records the number of walks that stay entirely within $T$. Now when we consider $\sum Aq^n$ all of the $q$ coefficients of $tot(q)$ and $bot(q)$ simply tell us how many steps in the future a particular walk would take from $T$, even if we did not actually have the steps for that particular walk in the graph. 

With the above observations, we can use the following transition matrix to keep track of the number of walks:

\begin{equation}\label{k1matrix}
M_1=\begin{bmatrix}
0& 0& 0&0&0\\
1&1&0&0&1\\
top(q)-bot(q)& top(q)-bot(q)&0&0&0\\
bot(q)&bot(q)&0&0&0\\
0&0&0&1&0
\end{bmatrix}
\end{equation}

Then by recalling that $(I-M_1q)^{-1}=\sum_{n\ge 0} M^nq^n$, we can invert this matrix in Mathematica and sum the first column (as done for finding the approximate growth rates earlier) to find the generating function. Similarly, this generating function only counts half of the walks of length $2$ or greater, so we double it and multiply by $q^2$ to obtain the final generating function:

\begin{equation}\label{kis2gen}
F_1(q)=1+q-2q^2\frac{1+bot(q)q^2+tot(q)q}{-1+q+bot(q)q^3}.
\end{equation}

The expansion of this generating function begins:

\[
F_1(q)=1+q+2q^2+4q^3+8q^4+14q^5+24q^6+40q^7+66q^8+106q^9+170q^{10}+...
\]
\section{Generating Function for $k=2$}
We would now like to repeat the developments of the above section for when $k=2$. We can do this with another transfer matrix and an explicit solution for walks in the corresponding infinite subgraph, which in this case consists of the descendants of the left descendants of $1234$ in the graph for $k=2$. The subgraph of interest in this case has structure similar to that from the $k = 1$ case, except now the edges that right descend return to nodes two levels below rather than one level below. We now need to keep track of the number of walks in the upper graph ending on $1245$ as well as those ending on $1256$, as each of these leads back into the original graph. Due to the more complicated structure of this graph we are unable to provide a simple generating function for the number of walks it contains, but if we define $tot'(q)$, $bot_1'(q)$ and $bot_2'(q)$ where $bot'_1(q)$ counts walks ending on $1245$ and $bot'_2(q)$ counts walks ending on $1256$ in the upper part of the graph, then we see that the generating function is:

\[
1+q-2q^2\frac{1+q+q^2+q^4(1+bot_2'(q))+q^3(1+tot(q))}{-1+q+q^2+q^4-q^7bot_2'(q)+q^5(bot_1'(q)+bot_2'(q))-q^6(bot_1'(q)+bot_2'(q))}
\]

To make use of this formula, one may manually compute the values for the generating functions that it contains. Note that solving $tot'(q)$, $bot_1'(q)$ and $bot_2'(q)$ are not part of the framework of \cite{Flajolet80} as they are not easily encoded as words over a simple alphabet.

\section{Future Work}
There are many questions that remain to be answered about locally convex permutations and words. A natural goal is to find methods of enumerating these permutations for $k\ge 3$, which is more difficult as the permutations can no longer be constructed in such a regular way. 

Another question is whether there exists an algebraic generating function in the case $k=2$?  For larger $k$? Furthermore, Conjecture \ref{conj} stands unsolved.

\section{Acknowledgments}

The research of both authors was supported by NSF Grant 1263009 and conducted during the Summer 2014 REU program at East Tennessee State University. We would like to thank Bill Kay for suggesting the problem and Anant Godbole for his suggestion about using subadditivity.

\bibliographystyle{plain}
\bibliography{forShi2}

\begin{thebibliography}{10}

\bibitem{Albert06}
M.~H. Albert, M.~Elder, A.~Rechnitzer, P.~Westcott, and M.~Zabrocki.
\newblock On the stanley-wilf limit of 4231-avoiding permutations and a
  conjecture of arratia.
\newblock {\em Advances in Applied Mathematics}, 36(2):96 -- 105, 2006.
\newblock Special Issue on Pattern Avoiding Permutations.

\bibitem{Albert11}
M.~H. Albert, S.~Linton, N.~Ru{\v{s}}kuc, V.~Vatter, and S.~Waton.
\newblock On convex permutations.
\newblock {\em Discrete Math.}, 311(8-9):715--722, 2011.

\bibitem{Arratia99}
R.~Arratia.
\newblock On the stanley-wilf conjecture for the number of permutations
  avoiding a given pattern.
\newblock {\em Electron. J. Combin}, 6(1):1--4, 1999.

\bibitem{Branden05}
P.~Br\"{a}nd\'{e}n and T.~Mansour.
\newblock Finite automata and pattern avoidance in words.
\newblock {\em Journal of Combinatorial Theory, Series A}, 110(1):127 -- 145,
  2005.

\bibitem{Flajolet80}
P.~Flajolet.
\newblock Combinatorial aspects of continued fractions.
\newblock {\em Annals of Discrete Mathematics}, 9:217--222, 1980.

\bibitem{Glasser87}
M.~Glasser, V.~Privman, and N.~Svrakic.
\newblock Temperley's triangular lattice compact cluster model: exact solution
  in terms of the $q$ series.
\newblock {\em Journal of Physics A: Mathematical and General}, 20(18):L1275,
  1987.

\bibitem{Kitaev11}
S.~Kitaev.
\newblock {\em Patterns in Permutations and Words}.
\newblock Springer, 2011.

\bibitem{Mansour01}
T.~Mansour and A.~Vainshtein.
\newblock Restricted permutations and {C}hebyshev polynomials.
\newblock {\em S\'em. Lothar. Combin.}, 47:Article B47c, 17, 2001/02.

\bibitem{Marcus04}
A.~Marcus and G.~Tardos.
\newblock Excluded permutation matrices and the stanley-wilf conjecture.
\newblock {\em Journal of Combinatorial Theory, Series A}, 107(1):153 -- 160,
  2004.

\bibitem{Odlyzko88}
A.~Odlyzko and H.~Wilf.
\newblock The editor's corner: $n$ coins in a fountain.
\newblock {\em American Mathematical Monthly}, pages 840--843, 1988.

\bibitem{Simion85}
R.~Simion and F.~W. Schmidt.
\newblock Restricted permutations.
\newblock {\em European Journal of Combinatorics}, 6(4):383 -- 406, 1985.

\bibitem{Stanley11}
R.~Stanley.
\newblock {\em Enumerative Combinatorics}, volume~1.
\newblock Cambridge University Press, 2011.

\bibitem{Vatter06}
V.~Vatter.
\newblock Finitely labeled generating trees and restricted permutations.
\newblock {\em Journal of Symbolic Computation}, 41(5):559 -- 572, 2006.

\bibitem{Wilf94}
H.~Wilf.
\newblock {\em Generatingfunctionology.}
\newblock Academic Press, Inc, 1994.

\end{thebibliography}

\end{document}